\documentclass[a4paper,11pt,reqno]{amsart}
\usepackage{amssymb,amsmath,hyperref}
\title[The mixed mock modularity of a new $U$-type function]{The mixed mock modularity of a new $U$-type function related to the  Andrews--Gordon identities}
\author{Nikolay E. Borozenets}
\address{Department of Mathematics and Computer Science, Saint Petersburg State University, Saint Petersburg, 199034, Russia}
\email{nikolayborozenets.spbumcs@gmail.com}

\newcommand\N{\mathbb{N}}
\newcommand\Z{\mathbb{Z}}
\newcommand\C{\mathbb{C}}
\renewcommand\theta{\vartheta}
\newcommand{\Mod}[1]{\ (\mathrm{mod}\ #1)}

\newcommand\myatop[2]{\left[{{#1}\atop#2}\right]}
\newcommand\sg{\operatorname{sg}}
\newtheorem{theorem}{Theorem}
\newtheorem{lemma}[theorem]{Lemma}
\newtheorem{corollary}[theorem]{Corollary}
\newtheorem{proposition}[theorem]{Proposition}
\theoremstyle{definition}
\newtheorem{definition}[theorem]{Definition}
\newtheorem{remark}[theorem]{Remark} 
\newtheorem{example}[theorem]{Example}

\numberwithin{theorem}{section} 
\numberwithin{equation}{section}

\makeatletter
\@namedef{subjclassname@2020}{\textup{2020} Mathematics Subject Classification}
\makeatother

\begin{document}

\date{6 October 2022}



\begin{abstract}
    In this paper we resolve a question by Bringmann, Lovejoy, and Rolen on a new vector-valued $U$-type function. We obtain an expression for a corresponding family of Hecke--Appell-type sums in terms of mixed mock modular forms; that is, we express the sum in terms of Appell functions and theta functions. This $U$-type function appears from considering the special polynomials related to generating functions for the partitions occurring in Gordon’s generalization of the Rogers--Ramanujan identities.
\end{abstract}

\maketitle

\section{Introduction and Statement of the Main Results}
Let $q$ be a nonzero complex number with $|q|<1$.  We recall the $q$-Pochhammer notation, defined by
\begin{equation*}
(x)_n=(x;q)_n:=\prod_{i=0}^{n-1}(1-xq^i), \ \ (x)_{\infty}=(x;q)_{\infty}:=\prod_{i\ge 0}(1-xq^i)
\end{equation*}
along with the Gaussian polynomials
\begin{equation*}
\myatop{n}{k}_q := \begin{cases}
   \frac{(q)_n}{(q)_k(q)_{n-k}}, & 0 \leq k \leq n\\
   0, & \text{otherwise}.
\end{cases}
\end{equation*}
We begin by introducing the special polynomials $H_n(t,m;b;q)$ for $t \in \mathbb{N}$, $1\leq m \leq t$, $b\in\{0,1\}$ according to the \cite{BLR}:
\begin{equation*}
H_n(t,m;b;q) := \sum_{n=n_t\geq \dots \geq n_1 \geq 0} \prod_{j=1}^{t-1}q^{n^2_j+(1-b)n_j}\myatop{n_{j+1}-n_j-bj+\sum_{r=1}^{j}(2n_r+\chi_{m>r})}{n_{j+1}-n_j}_q.
\end{equation*}
Here we use the usual characteristic function $\chi_A$, defined to be $1$ if $A$ is true and $0$ otherwise.
These polynomials are related to generating functions for the partitions occurring in Gordon’s generalization of the Rogers--Ramanujan identities \cite{W}. To describe it let $G_{k,i,i',L}(q)$ be the generating function for partitions of the form
\begin{equation*}
\sum_{j=1}^{L-1}jf_j
\end{equation*}
with frequency conditions $f_1\leq i-1$, $f_{L-1}\leq i'-1$, $f_i+f_{i+1}\leq k$ for $1\leq k \leq L-2$. These conditions are a finitization of the conditions in Gordon partition identity \cite[Theorem 3]{W}. In this way the corresponding analytic identity by Andrews \cite[Theorem 4]{W} 
\begin{equation*}
\sum_{n_1,\dots,n_k \geq 0} \frac{q^{N^2_1+\cdots+N^2_k+N_i+\cdots+N_k}}{(q)_{n_1}\cdots(q)_{n_k}} = \frac{1}{(q)_{\infty}}\sum_{j\in \Z}(-1)^j
q^{j((2k+3)(j+1)-2i)/2}
\end{equation*}
where $N_j = n_j+\cdots+n_k$ appears to be a limiting case of some polynomial identities. By making some changes of variables and comparing with \cite[Theorem 5]{W}, it can be shown that
\begin{equation*}
H_n(t,m;b;q^{-1}) = q^{(t-1)bn-2(t-1)\binom{n+1}{2}} G_{t-1,m,t,2n-b+1}(q).
\end{equation*} 

In this paper we will study the following function found in a recent paper of Bringmann, Lovejoy, and Rolen \cite[Section 4]{BLR}.

\begin{definition} \label{def:Umt} For $x\in \C \smallsetminus \{0\}$ and $1 \leq m \leq t$ we define
\[ \mathcal{U}_{t}^{(m)}(x;q) := \sum_{n=0}^{\infty}q^n(-x)_n(-q/x)_nH_n(t,m;0;q). \]
\end{definition}

As stated in \cite[Section 4]{BLR}, one can use Bailey pair methods to find the following Hecke--Appell-type series.

\begin{theorem} \cite[Section $4$, Question $3$]{BLR} \label{theorem:HTformulaforU}
We have
\begin{align*}
\mathcal{U}_{t}^{(m)}(-x;q) &= \frac{(x)_{\infty}(\frac{q}{x})_{\infty}}{(q)^2_{\infty}}\left(\sum_{\substack{r,s, u\geq 0 \\ r\equiv s \Mod 2}}+\sum_{\substack{r,s,u < 0 \\ r\equiv s \Mod 2}}\right)(-1)^{\frac{r-s}{2}}x^u q^{\frac{r^2}{8}+\frac{4t+3}{4}rs+\frac{s^2}{8}+\frac{4t+3-2m}{4}r+\frac{1+2m}{4}s + \frac{r+s}{2}u }\\
&= \frac{(x)_{\infty}(\frac{q}{x})_{\infty}}{(q)^2_{\infty}}\left(\sum_{\substack{r,s\geq 0 \\ r\equiv s \Mod 2}}-\sum_{\substack{r,s < 0 \\ r\equiv s \Mod 2}}\right)\frac{(-1)^{\frac{r-s}{2}}q^{\frac{r^2}{8}+\frac{4t+3}{4}rs+\frac{s^2}{8}+\frac{4t+3-2m}{4}r+\frac{1+2m}{4}s}}{1-xq^{\frac{r+s}{2}}},
\end{align*}
where the second identity holds for $|q| < |x| < 1$.
\end{theorem}

The function $\mathcal{U}_{t}^{(m)}(x;q)$ is similar to the function $U_{t}^{(m)}(x;q)$ studied by Hikami and Lovejoy \cite{HL} and Mortenson and Zwegers \cite{MZ}:
\begin{equation*}
U_{t}^{(m)}(x;q) := q^{-t}\sum_{n=1}^{\infty}q^n(-xq)_{n-1}(-q/x)_{n-1} H_n(t,m;1;q).
\end{equation*}

In \cite{HL}, one finds that $U_{t}^{(m)}(-1;q)$ are vector-valued quantum modular forms which are in a sense dual to a certain generalization of the Kontsevich-Zagier function.  Hikami and Lovejoy \cite{HL} also expressed $U_{t}^{(m)}(x;q)$ in terms of a family of Hecke--Appell-type sums analogous to Theorem \ref{theorem:HTformulaforU}, see \cite[Theorem 5.6]{HL}.  Hikami and Lovejoy also asked what kind of modular behavior is expressed by the Hecke--Appell-type sum.  Motivated by the work \cite{HL}, Bringmann, Lovejoy, and Rolen \cite[Section 4]{BLR} posed the following three questions:

\begin{enumerate}
    \item What sort of modular behavior is implied by the Hecke--Appell-type sum in Theorem \ref{theorem:HTformulaforU}?
    \item Are the functions $\mathcal{U}_{t}^{(m)}(-1;q)$ quantum modular forms?
    \item Are the functions $\mathcal{U}_{t}^{(m)}(-1;q)$ related to some sort of Kontsevich--Zagier-type series?
\end{enumerate}

Recently, Mortenson and Zwegers resolved the question of Hikami and Lovejoy and expressed the family of Hecke--Appell-type sums \cite[Theorem 5.6]{HL} in terms of a family of mixed mock modular forms \cite[Theorem 1.7]{MZ}.  Mixed mock modular forms are functions that lie in the tensor space of mock modular forms and modular forms \cite{DMZ}.   In this paper, we will use the methods of \cite{MZ} to answer Bringmann, Lovejoy, and Rolen's first question about $\mathcal{U}_{t}^{(m)}(x;q)$. Questions 2 and 3 will be the focus of future research by the author.

To state our results, we first recall some standard notation:
\begin{equation*}\sg(r,s):=\frac{\sg(r)+\sg(s)}{2}, \ \textup{where} \ 
\sg(r)=\begin{cases} 1 & \text{if}\ r\geq 0,\\ -1 & \text{if}\ r<0.
\end{cases}\end{equation*}
For theta functions, we write
\begin{equation*}
\Theta(x;q):=(x)_{\infty}(q/x)_{\infty}(q)_{\infty}=\sum_{n=-\infty}^{\infty}(-1)^nq^{\binom{n}{2}}x^n.
\end{equation*}
To determine the modularity of $\mathcal{U}_{t}^{(m)}(x;q)$, we will mainly consider the two following functions:
\begin{definition} We define
\begin{equation*}
    g_{t,m}(x) := \sum_{r\equiv s \Mod 2} \sg(r,s) \frac{(-1)^{\frac{r-s}{2}}q^{\frac{r^2}{8}+\frac{4t-1}{4}rs+\frac{s^2}{8}+\frac{4t-1-2m}{4}r+\frac{1+2m}{4}s}}{1-xq^{\frac{r+s}{2}}}
\end{equation*}
and
\begin{equation*}
f_{t,m}(x) := \frac{\Theta(x;q)}{(q)^3_{\infty}}g_{t,m}(x).
\end{equation*}
\end{definition}

\begin{remark} \label{rem:UandF}
Theorem \ref{theorem:HTformulaforU} thus gives
\begin{equation*}
 \mathcal{U}_t^{(m)}(-x;q)= f_{t+1,m}(x),
\end{equation*}
which has as a special case
\begin{equation*}
 \mathcal{U}(-x;q):=\mathcal{U}_{1}^{(1)}(-x;q)= f_{2,1}(x).
\end{equation*}
\end{remark}

Integral to the statement of our main result is the following indefinite binary theta series:

\begin{definition}\label{def:indbinthetaser}
Let $t\in\N$ be fixed.
For $p,m\in\Z$ we define
\begin{equation*}
\theta_{p,m} := \sum_{r\equiv s \Mod 2} \sg(r,s) (-1)^{\frac{r-s}{2}}q^{B_{p,m}(r,s)}
\end{equation*}
and
\begin{equation*}
\theta^{*}_{p,m} := q^{C_{p,m}}\theta_{p,m}^*,
\end{equation*} 
where
\begin{equation*}
B_{p,m}(r,s) := \frac{r^2}{8}+\frac{4t-1}{4}rs+\frac{s^2}{8}+\frac{4p-1-2m}{4}r+\frac{1+2m}{4}s\\
\end{equation*}
and
\begin{equation*}
C_{p,m} := \frac{-2p^2-t(2m+1)(2m+1-4p)}{8t(2t-1)}. 
\end{equation*} 
\end{definition}

Now we are able to formulate the main result of the paper.

\begin{theorem}\label{theorem:main}
For $t\geq 2$ and $1\leq m<t$ we have
\begin{align*}
f_{t,m}(x) = \frac{q^{-\frac{t}{4}}}{(q)^3_{\infty}} \sum_{k\operatorname{mod} 2t}(-1)^k q^{\frac{(k+t)^2}{4t}}\theta_{t+k,m+k} \sum_{\substack{r,l \in \Z\\ l\equiv k \Mod {2t}}} \sg(r,l)(-1)^r q^{\frac{r^2}{2}+rl+\frac{2t-1}{4t}l^2+\frac{1}{2}r}x^{-r}.
\end{align*}
\end{theorem}

The proof of Theorem \ref{theorem:main} is presented in Section \ref{sect:proofmainth}. Prior to it the properties of indefinite binary theta series $\theta_{p,m}$ are developed in Section \ref{sect:propbinser} and the functional equation for $f_{t,m}$ is derived in Section \ref{sect:functeqF}. 

We will use Theorem \ref{theorem:main} to derive the modular properties of $(q)_{\infty}^3\mathcal{U}^{(m)}_{t}(x;q)$ in Section \ref{sect:applicationmainmt}.  We introduce the standard definition for Hecke-type double-sums:
\begin{equation}\label{eq:fabc}
f_{a,b,c}(x,y;q) := \sum_{r,s \in \Z} \sg(r,s) (-1)^{r+s} x^r y^s q^{a\binom{r}{2}+brs+c\binom{s}{2}}
\end{equation}
where $a$, $b$, $c$ are positive integers with $b^2-ac>0$.  In terms of the above definition, we obtain an interpretation of $\mathcal{U}^{(m)}_{t}(x;q)$ in terms of sums of binary products of Hecke-type sums.
\begin{corollary} \label{cor:main}
We have 
\begin{align*}
(q)^3_{\infty}&f_{t,m}(x) \\
&= \sum_{k=0}^{2t-1}(-1)^k q^{\binom{k}{2}+k} (f_{1,4t-1,1}(q^{2t+k-m}, q^{1+m+k}; q) + q^{2t+k}f_{1,4t-1,1}(q^{4t+k-m}, q^{2t+m+k+1}; q)) \\
&\qquad \times f_{1,2t,2t(2t-1)}(x^{-1}q^{k+1}, -q^{(k+t)(2t-1)}; q).
\end{align*}
\end{corollary}
One can then use \cite[Theorem 1.3]{HM} or \cite[Theorem 1.8, Corollary 4.2]{MZ} to evaluate the double-sums in terms of theta functions and Appell functions, where we use the following definition for an Appell function
\begin{equation}\label{eq:appell}
m(x,z;q):=\frac{1}{\Theta(z;q)}\sum_{r=-\infty}^{\infty}\frac{(-1)^rq^{\binom{r}{2}}z^r}{1-q^{r-1}xz}.
\end{equation}

In Section \ref{sect:anotherforF}, we use a different method to solve for the functional equation for $f_{t,m}(x)$, and we obtain another formula.
\begin{theorem}\label{theorem:newcalc} For $t\geq 2$ and $1\leq m<t$ we have
\[\begin{split}
f_{t,m}(x) = x^{-2t+m+1}& m(x^{-2t+1}  q^{m},-x^{2t-1}; q^{2t-1})  +\\+x^{-m}  & m(x^{-2t+1}q^{2t-1-m} ,-x^{2t-1};q^{2t-1})  \\
&- \frac{x^{-2t}}{\Theta(-x^{2t-1};q^{2t-1})(q)_{\infty}} \sum_{k=0}^{2t-1} x^k q^{k}\theta_{t+k,m+k} \\& \times\sum_{s=0}^{2t-1} (-1)^s q^{\binom{s}{2}}\Theta(q^{(2t-1)(t-s)}; q^{2t(2t-1)}) x^s  \Theta(q^sx^{2t};q^{2t})m(-q^{k}x^{-2t}, q^s x^{2t}; q^{2t}).
\end{split}\]
\end{theorem}
\noindent Theorem \ref{theorem:newcalc} also gives an interpretation of $f_{t,m}(x)$ in terms of theta functions and Appell functions.

We also consider special case $\mathcal{U}(x;q)$ and study its modular properties. In Section \ref{sect:applicationmain11} we evaluate $\mathcal{U}(x;q)$ in terms of Hecke-type sums:
\begin{equation*}
(q)_{\infty}\mathcal{U}(-x;q) = f_{1,2,3}(x^{-1}q^2,q^3;q).
\end{equation*}
In Section \ref{sect:applicationnewf123} we apply Theorem \ref{theorem:newcalc} to find the Appell form and the ``mod theta'' congruence for $f_{1,2,3}(x^{-1}q^2,q^3;q)$.  In Section \ref{sect:nothetaid}, we find another expression for the mock modularity of $f_{1,2,3}(x^{-1}q^2,q^3;q)$ by using identities found in \cite{M1}.

\section{Properties of indefinite binary theta series}\label{sect:propbinser}
In this section we will consider properties of the theta functions $\theta_{p,m}$ and $\theta_{p,m}^*$ as in Definition \ref{def:indbinthetaser}.  

\begin{proposition}\label{prop:thetaproperties}
We have
\begin{enumerate}
    \item[\textnormal{(a)}] $\theta_{p,m} = q^{p+t}\theta_{p+2t,m+2t}$, $\theta_{p+2t,m+2t}^* = \theta_{p,m}^*$;
    \item[\textnormal{(b)}] $\theta_{p,m} = -q^{p-m-t}\theta_{p,m+2t-1}$, $\theta_{p,m}^* = -\theta_{p,m+2t-1}^*$;
    \item[\textnormal{(c)}] $\theta_{p,m} = -\theta_{-p,-m-1}$, $\theta_{p,m}^* =-\theta_{-p,-m-1}^*$;
    \item[\textnormal{(d)}] $\theta_{p,m} = \theta_{p,2p-m-1}$, $\theta_{p,m}^* = \theta_{p,2p-m-1}^*$.
\end{enumerate}
\end{proposition}
We first present a lemma which is helpful in proof of Proposition \ref{prop:thetaproperties}.
\begin{lemma}\label{lemma:thetaserieseq0}
For a fixed $s\in\Z$ we have
\begin{equation*}
\sum_{\substack{r \in \Z \\ r\equiv s \Mod 2}}(-1)^{\frac{r-s}{2}}q^{B_{p,m}(r,s)} = 0.
\end{equation*}
For a fixed $r\in\Z$ we have
\begin{equation*}
\sum_{\substack{s \in \Z \\ s\equiv r \Mod 2}}(-1)^{\frac{r-s}{2}}q^{B_{p,m}(r,s)} = 0.
\end{equation*}
\end{lemma}

Now we present the proofs of both Lemma \ref{lemma:thetaserieseq0} and Proposition \ref{prop:thetaproperties}: 
\begin{proof}[Proof of Lemma \ref{lemma:thetaserieseq0}]
To obtain the first identity it is needed to change the variable
\begin{equation*}
r \to -r-2(4t-1)s-2(4p-1-2m)
\end{equation*}
and apply the following property of the binary quadratic form $B_{p,m}(r,s)$:
\begin{equation*}
B_{p,m}(-r-2(4t-1)s-2(4p-1-2m),s) = B_{p,m}(r,s). \\
\end{equation*}
Similarly for the second identity, change
\begin{equation*}
s \to -s-2(4t-1)r-2(1+2m)
\end{equation*}
and use
\begin{equation*}
B_{p,m}(r,-s-2(4t-1)r-2(1+2m)) = B_{p,m}(r,s).\qedhere
\end{equation*}
\end{proof}
\begin{proof}[Proof of Proposition \ref{prop:thetaproperties}]
First we set the notation:
\begin{equation*}
\delta_m(r):=\begin{cases} 1 &\text{if}\ r=m,\\ 0 &\text{otherwise.}\end{cases}
\end{equation*}
To prove (a) we use the identity 
\begin{equation*}
\sg(r,s) = \sg(r-1,s-1)+\delta_0(r)+\delta_0(s).
\end{equation*}
So we have
\begin{align*}
\theta_{p,m} &= \sum_{r\equiv s \Mod 2} \sg(r,s) (-1)^{\frac{r-s}{2}}q^{B_{p,m}(r,s)}\\ &=\sum_{r\equiv s \Mod 2} \sg(r-1,s-1) (-1)^{\frac{r-s}{2}}q^{B_{p,m}(r,s)} \\
 &\qquad \qquad + \sum_{s\equiv 0 \Mod 2} (-1)^{\frac{-s}{2}}q^{B_{p,m}(0,s)} + \sum_{r\equiv 0 \Mod 2} (-1)^{\frac{r}{2}}q^{B_{p,m}(r,0)}.
\end{align*}
Using Lemma \ref{lemma:thetaserieseq0} and making the change of variables $(r,s) \to (r+1,s+1)$ yields
\begin{equation*}
\theta_{p,m} = \sum_{r\equiv s \Mod 2} \sg(r,s) (-1)^{\frac{r-s}{2}}q^{B_{p,m}(r+1,s+1)}.
\end{equation*}
Apply the property
\begin{equation} \label{eq:propertyBplus1}
B_{p,m}(r+1,s+1) = B_{p,m}(r,s)+tr+ts+t+p = B_{p+2t,m+2t}(r,s)+t+p
\end{equation}
to obtain (a) for the $\theta_{p,m}$. To obtain (a) for $\theta_{p,m}^*$ use the property
\begin{equation*}
C_{p+2t,m+2t} = C_{p,m} + t + p.
\end{equation*}
Other identities are proved similarly. For (b) use
\begin{equation*}
\sg(r,s)=\sg(r-1,s+1)+\delta_0(r)-\delta_0(s+1)
\end{equation*}
and properties
\begin{gather*}
B_{p,m}(r+1,s-1) = B_{p,m}(r,s)-\frac{2t-1}{2}r+\frac{2t-1}{2}s+p-m-t = B_{p,m+2t-1}(r,s)+p-m-t,\\
C_{p,m} = C_{p,m+2t-1} + m-p+t.
\end{gather*}
For (c) use 
\begin{equation*}
\sg(-r,-s) = -\sg(r,s)+\delta_0(r)+\delta_0(s)
\end{equation*}
as well as
\begin{gather*}
B_{p,m}(-r,-s) = B_{-p,-m-1}(r,s),\\
C_{-p,-m-1} = C_{p,m}.
\end{gather*}
For (d) use 
\begin{equation*}
\sg(r,s)=\sg(s,r)
\end{equation*}
as well as
\begin{gather*}
B_{p,m}(r,s) = B_{p, 2p-m-1}(s,r),\\
C_{p,m} = C_{p,2p-m-1}.\qedhere
\end{gather*}
\end{proof}

\section{Functional equation for $f_{t,m}(x)$}\label{sect:functeqF}
In this section we derive the functional equation for $f_{t,m}(x)$ which is used in the proof of Theorem \ref{theorem:main} and in the proof of Theorem \ref{theorem:newcalc}.
\begin{proposition}\label{prop:funceqGF}
We have
\[ f_{t,m}(qx) = -x^{2t-1}f_{t,m}(x) + (x^{2t-m-1}+x^{m}) - \frac{\Theta(x;q)}{(q)^3_{\infty}}\sum_{k=0}^{2t-1}x^{k-1}q^k \theta_{t+k,m+k}. \]
\end{proposition}
\begin{proof}[Proof of Proposition \ref{prop:funceqGF}]
Let us consider
\begin{equation} \label{eq:specialsumforprooffunceq}
\sum_{r\equiv s \Mod 2}\sg(r,s) (-1)^{\frac{r-s}{2}}q^{\frac{r^2}{8}+\frac{4t-1}{4}rs+\frac{s^2}{8}+\frac{4t-1-2m}{4}r+\frac{1+2m}{4}s} \frac{1-x^{2t}q^{t(r+s+2)}}{1-xq^{\frac{r+s+2}{2}}}.
\end{equation}
Let us calculate this sum in two ways. The first way is to open the brackets of the last factor:
\begin{align*}
g_{t,m}(qx) & - x^{2t}q^{2t} \sum_{r\equiv s \Mod 2}\sg(r,s) \frac{(-1)^{\frac{r-s}{2}}q^{\frac{r^2}{8}+\frac{4t-1}{4}rs+\frac{s^2}{8}+\frac{4t-1-2m}{4}r+\frac{1+2m}{4}s+tr+ts}}{1-xq^{\frac{r+s+2}{2}}}\\
 &\qquad = g_{t,m}(qx) - x^{2t}q^{2t} \sum_{r\equiv s \Mod 2}\sg(r,s) \frac{(-1)^{\frac{r-s}{2}}q^{B_{t,m}(r,s)+tr+ts}}{1-xq^{\frac{r+s+2}{2}}}.
\end{align*}
Change the variables $(r,s) \to (r-1,s-1)$ and recall property (\ref{eq:propertyBplus1}) from the previous section:
\begin{align*}
g_{t,m}(qx) &- x^{2t} \sum_{r\equiv s \Mod 2}\sg(r-1,s-1) \frac{(-1)^{\frac{r-s}{2}}q^{B_{t,m}(r,s)}}{1-xq^{\frac{r+s}{2}}}\\
&\qquad =
g_{t,m}(qx) - x^{2t} \sum_{r\equiv s \Mod 2}\sg(r-1,s-1) \frac{(-1)^{\frac{r-s}{2}}q^{\frac{r^2}{8}+\frac{4t-1}{4}rs+\frac{s^2}{8}+\frac{4t-1-2m}{4}r+\frac{1+2m}{4}s}}{1-xq^{\frac{r+s}{2}}}.
\end{align*}
Because $\sg(r-1,s-1) = \sg(r,s)-\delta_0(r)-\delta_0(s)$, we obtain
\begin{align*}
g_{t,m}(qx) &- x^{2t} \left(g_{t,m}(x) - \sum_{s \equiv 0 \Mod 2} \frac{(-1)^{\frac{-s}{2}}q^{\frac{s^2}{8}+\frac{1+2m}{4}s}}{1-xq^{\frac{s}{2}}}  - \sum_{r \equiv 0 \Mod 2} \frac{(-1)^{\frac{r}{2}}q^{\frac{r^2}{8}+\frac{4t-1-2m}{4}r}}{1-xq^{\frac{r}{2}}} \right) \\
&\qquad = g_{t,m}(qx) - x^{2t} \left(g_{t,m}(x) - \sum_{s \in \Z} \frac{(-1)^{s}q^{\frac{s^2}{2}+(m+\frac{1}{2})s}}{1-xq^{s}}  - \sum_{r \in \Z} \frac{(-1)^{r}q^{\frac{r^2}{2}+(2t-1-m+\frac{1}{2})r}}{1-xq^{r}} \right).
\end{align*}
We know the formula
\begin{equation*} 
\sum_{k \in \Z} \frac{(-1)^k q^{\frac{k^2}{2}+(n+\frac{1}{2})k}}{1-xq^k} = \frac{(q)^3_{\infty}}{x^n \Theta(x;q)}.
\end{equation*}
So we find
\begin{equation*}
g_{t,m}(qx) - x^{2t} \left(g_{t,m}(x) - \frac{(q)^3_{\infty}}{x^{m} \Theta(x;q)}  - \frac{(q)^3_{\infty}}{x^{2t-m-1} \Theta(x;q)} \right).
\end{equation*}

The second way to consider (\ref{eq:specialsumforprooffunceq}) is to insert
\begin{equation*}
\frac{1-x^{2t}q^{t(r+s+2)}}{1-xq^{\frac{r+s+2}{2}}} = \sum_{k=0}^{2t-1} x^kq^{\frac{r+s+2}{2} k}.
\end{equation*}
Hence
\begin{align*}
\sum_{r\equiv s \Mod 2}&\sg(r,s) (-1)^{\frac{r-s}{2}}q^{\frac{r^2}{8}+\frac{4t-1}{4}rs+\frac{s^2}{8}+\frac{4t-1-2m}{4}r+\frac{1+2m}{4}s} \sum_{k=0}^{2t-1} x^kq^{\frac{r+s+2}{2} k}\\
&= \sum_{k=0}^{2t-1} x^k q^k \sum_{r\equiv s \Mod 2}\sg(r,s) (-1)^{\frac{r-s}{2}}q^{\frac{r^2}{8}+\frac{4t-1}{4}rs+\frac{s^2}{8}+\frac{4t-1-2m+2k}{4}r+\frac{1+2m+2k}{4}s}\\
&= \sum_{k=0}^{2t-1} x^k q^k \theta_{t+k, m+k}.
\end{align*}
The result is the functional equation for $g_{t,m}(x)$:
\begin{equation*}
\sum_{k=0}^{2t-1} x^k q^k \theta_{t+k, m+k} = g_{t,m}(qx) - x^{2t} \left(g_{t,m}(x) - x^{-m}\frac{(q)^3_{\infty}}{\Theta(x;q)}  - x^{-2t+m+1}\frac{(q)^3_{\infty}}{ \Theta(x;q)} \right).
\end{equation*}
To obtain the functional equation for $f_{t,m}(x)$ we need to multiply the above by $-\frac{\Theta(x;q)}{x(q)^3_{\infty}} = \frac{\Theta(qx;q)}{(q)^3_{\infty}}$ and rearrange terms
\begin{equation*}
f_{t,m}(qx) = -x^{2t-1}f_{t,m}(x) + (x^{2t-m-1}+x^{m}) - \frac{\Theta(x;q)}{(q)^3_{\infty}}\sum_{k=0}^{2t-1}x^{k-1}q^k \theta_{t+k,m+k}.\qedhere
\end{equation*}
\end{proof}

\begin{example}
Let us consider the case $t=1$ and arbitrary $m \in \Z$. Using properties from Proposition \ref{prop:thetaproperties} we can obtain $\theta_{1,m}=0$ and $\theta_{2,m+1}=0$. Thus, the functional equation for $f_{t,m}(x)$ takes the form:
\begin{equation}
f_{1,m}(qx) = -x f_{1,m}(x) + x^{1-m}+x^{m}.\label{equation:f-fnq}
\end{equation}
As $f_{1,m}(x)$ has no poles, we can calculate the solution of this equation by considering the expansion of $f_{1,m}(x)$ into the following modified Laurent series:
\begin{equation}
f_{1,m}(x) = \sum_{r \in \Z}(-1)^r q^{-\binom{r}{2} - r}a_r x^r.
\label{equation:f-laurent}
\end{equation}
We substitute the expansion (\ref{equation:f-laurent}) into the functional equation (\ref{equation:f-fnq}). We can equate the coefficients of $x^r$ on both sides and divide them by $(-1)^r q^{-\binom{r}{2}}$.
\begin{equation*}
a_r = a_{r-1} + (-1)^r q^{\binom{r}{2}} \delta_{1-m}(r) + (-1)^r q^{\binom{r}{2}} \delta_{m}(r).
\end{equation*}
Thus, we have
\begin{equation*}a_r = 
\begin{cases}
   (-1)^{m+1} q^{\binom{m}{2}} & \textup{if} \ r \in \{1-m, m-1\},\\
   0 & \textnormal{otherwise}.
\end{cases}
\end{equation*}
We conclude that
\begin{equation*}
f_{t,m}(x) = (-1)^{m+1} q^{\binom{m}{2}}\sum_{r = 1-m}^{m-1}(-1)^r q^{-\binom{r}{2} - r} x^r.
\end{equation*}
\end{example}

\section{Proof of Theorem \ref{theorem:main}}\label{sect:proofmainth}
Now we prove Theorem \ref{theorem:main} by using the functional equation for $f_{t,m}(x)$. As a main tool for doing it we will use the following useful lemma:
\begin{lemma}\label{lemma:solverectool}
Let $(a_r)_{r \in \Z}$ be a sequence such that
\[ a_r-a_{r+d} = b_r + C_r, \]
with $|b_r| < Aq^{B|r|}$ $\forall r \in \Z$ with some positive $A,B \in \mathbb{R}$ and $C_r=0$ if $r \notin [-d, -1]$.
Then \[ a_r = \sum_{l \in \Z} \sg(r,l)b_{r+ld}. \]
\end{lemma}
\begin{proof}[Proof of Lemma \ref{lemma:solverectool}]
Define the following sequence:
\begin{equation*}
\widehat{a}_r = \sum_{l \in \Z} \sg(r,l)b_{r+ld}.
\end{equation*}
Then we have
\begin{align*}
\widehat{a}_r - \widehat{a}_{r+d} &= \sum_{l \in \Z} (\sg(r,l)-\sg(r+d,l-1))b_{r+ld} \\
&= \sum_{l \in \Z}(\delta_0(l)-\delta_0(r+1)-\cdots-\delta_0(r+d))b_{r+ld}\\
&=  b_r - (\delta_0(r+1)+\cdots+\delta_0(r+d))\sum_{n \equiv r \Mod d}b_n.
\end{align*}
Note that $\lim_{r \rightarrow \pm \infty}a_r = 0$. This fact can be derived from the recursion identity for $a_r$ and the lemma condition on $b_r$. From this property we have
\begin{equation*}
\sum_{n \equiv r \Mod d}b_n = \sum_{n \equiv r \Mod d} a_n-a_{n+d} - C_n  = -\sum_{n \equiv r \Mod d} C_n.
\end{equation*}
Because $C_r=0$ if $r \notin [-d, -1]$ we have
\begin{equation*}
\widehat{a}_r - \widehat{a}_{r+d} =  b_r + (\delta_0(r+1)+\cdots+\delta_0(r+d))\sum_{n \equiv r \Mod d} C_n = b_r+C_r.
\end{equation*}
We see that $\widehat{a}_r$ satisfied the same recurrence relation as $a_r$.  Hence the expression $\widehat{a}_r-a_r$ is $d$-periodic.  Also we can notice that $\lim_{r \rightarrow \infty}\widehat{a}_r = 0$. From these facts we obtain $\widehat{a}_r-a_r = 0$.
\end{proof}

\begin{proof}[Proof of Theorem \ref{theorem:main}]
As $f_{t,m}(x)$ does not have poles,  consider the Laurent series in $x\in \C \smallsetminus \{0\}$ for $f_{t,m}(x)$ in the special form:
\begin{equation*}
f_{t,m}(x) = \sum_{r \in \Z} (-1)^r q^{-\frac{1}{2(2t-1)}r^2 +\frac{1}{2}r} a_r x^{-r}.
\end{equation*}
Recall Proposition \ref{prop:funceqGF} and the definition of the theta function:
\begin{align*}
f_{t,m}(qx) &= -x^{2t-1}f_{t,m}(x) + (x^{2t-m-1}+x^{m}) - \frac{\Theta(x;q)}{(q)^3_{\infty}}\sum_{k=0}^{2t-1}x^{k-1}q^k \theta_{t+k,m+k}\\&= -x^{2t-1}f_{t,m}(x) + (x^{2t-m-1}+x^{m}) - \frac{1}{(q)^3_{\infty}}\sum_{l\in \Z} (-1)^l q^{\binom{l}{2}} x^l \sum_{k=0}^{2t-1}x^{k-1}q^k \theta_{t+k,m+k}.
\end{align*} 
Making the change of variable $l \to l -k+1$ produces
\begin{equation*}
f_{t,m}(qx) = -x^{2t-1}f_{t,m}(x) + (x^{2t-m-1}+x^{m}) - \frac{1}{(q)^3_{\infty}}\sum_{l\in \Z} \sum_{k=0}^{2t-1} (-1)^{l-k+1} q^{\binom{l-k+1}{2}+k} \theta_{t+k,m+k} x^l.
\end{equation*}
Using the identity
\begin{equation*}
\binom{a-b}{2} = \binom{a}{2} - ab + \binom{b+1}{2},
\end{equation*}
we conclude
\begin{equation*}
f_{t,m}(qx) = -x^{2t-1}f_{t,m}(x) + (x^{2t-m-1}+x^{m}) + \frac{1}{(q)^3_{\infty}}\sum_{l\in \Z} (-1)^{l} q^{\binom{l}{2}} \sum_{k=0}^{2t-1} (-1)^{k} q^{\binom{k}{2}+k-l(k-1)} \theta_{t+k,m+k} x^l.
\end{equation*}
Considering the coefficient of $x^{-r}$ of both sides yields
\begin{equation*}
a_r = a_{r+2t-1} + b_r + C_r,
\end{equation*}
\begin{equation}\label{eq:seqb}
b_r =  q^{\frac{1}{2(2t-1)}r^2 + \frac{1}{2}r + \binom{r+1}{2}} \frac{1}{(q)^3_{\infty}} \sum_{k=0}^{2t-1} (-1)^{k} q^{\binom{k}{2}+k+r(k-1)} \theta_{t+k,m+k},
\end{equation}
\begin{equation}\label{eq:seqc}
C_r = \begin{cases}
   (-1)^{r} q^{\frac{1}{2(2t-1)}r^2 + \frac{1}{2}r}, & r \in \{-m, -2t+m+1\}\\
   0, & r \notin \{-m, -2t+m+1\}.
\end{cases}
\end{equation}
From here we can see that we are able to apply Lemma \ref{lemma:solverectool} with $d=2t-1$ as $b_r$ and $C_r$ satisfy the lemma conditions. Hence we have
\begin{equation*}
a_r = \sum_{l \in \Z} \sg(r, l)b_{r+l(2t-1)}.
\end{equation*}
Using the definitions of $b_r$, we obtain
\begin{align*}
f_{t,m}(x) &= \sum_{r,l \in \Z}\sg(r,l)(-1)^r q^{-\frac{1}{2(2t-1)}r^2+\frac{1}{2}r}x^{-r}b_{r+l(2t-1)}\\
&= \sum_{r,l \in \Z}\sg(r,l)(-1)^r q^{-\frac{1}{2(2t-1)}r^2+\frac{1}{2}r} 
\times q^{\frac{1}{2(2t-1)}(r+l(2t-1))^2 + \frac{1}{2}(r+l(2t-1)) + \binom{r+l(2t-1)+1}{2}}\\
&\qquad  \times \frac{1}{(q)^3_{\infty}} \sum_{k=0}^{2t-1} (-1)^{k} q^{\binom{k}{2}+k+(r+l(2t-1))(k-1)} \theta_{t+k,m+k}x^{-r}.
\end{align*}
Interchanging sums, we find
\begin{align}\label{eq:anotherformmain}
f_{t,m}(x) &= \frac{1}{(q)^3_{\infty}} \sum_{k=0}^{2t-1}(-1)^k q^{\binom{k}{2}+k}\theta_{t+k,m+k} \\
&\qquad \times \sum_{r,l \in \Z} \sg(r,l)(-1)^rq^{\binom{r}{2}+2trl+2t(2t-1)\binom{l}{2}+r(k+1)+l((k+t)(2t-1))}x^{-r}.
\notag
\end{align}
In conclusion, we have
\begin{align*}
f_{t,m}(x) &= \frac{1}{(q)^3_{\infty}} \sum_{k=0}^{2t-1}(-1)^k q^{\binom{k}{2}+k}\theta_{t+k,m+k} \sum_{r,l \in \Z} \sg(r,l)(-1)^r q^{\frac{r^2}{2}+2trl+2t(2t-1)\frac{l^2}{2}+r(k+\frac{1}{2})+l(2tk-k)}x^{-r}\\
&=
\frac{1}{(q)^3_{\infty}} \sum_{k=0}^{2t-1}(-1)^k q^{\binom{k}{2}+k-\frac{2t-1}{4t}k^2}\theta_{t+k,m+k} \sum_{r,l \in \Z} \sg(r,l)(-1)^r q^{\frac{r^2}{2}+r(2tl+k)+\frac{2t-1}{4t}(2tl+k)^2+\frac{1}{2}r}x^{-r}\\
&=
\frac{q^{-\frac{t}{4}}}{(q)^3_{\infty}} \sum_{k\operatorname{mod} 2t}(-1)^k q^{\frac{(k+t)^2}{4t}}\theta_{t+k,m+k} \sum_{\substack{r,l \in \Z\\ l\equiv k\Mod{2t}}} \sg(r,l)(-1)^r q^{\frac{r^2}{2}+rl+\frac{2t-1}{4t}l^2+\frac{1}{2}r}x^{-r}.\qedhere
\end{align*}
\end{proof}

\section{Proof of the Theorem \ref{theorem:newcalc}}\label{sect:anotherforF}
In this section we again obtain a closed expression for $f_{t,m}(x)$, but we use a different method. Firstly introduce modified theta functions $V_t(x)$ and $Y_t(x)$ and observe some of its properties.
\begin{definition} Let $t\in\N$ be fixed. We define
\begin{equation*}
V_{t}(x) := \Theta(-x^{2t-1};q^{2t-1}) = \sum_{r = -\infty}^{\infty} q^{(2t-1)\binom{r}{2}}x^{(2t-1)r}
\end{equation*}
and
\begin{equation*}
Y_{t}(x) := V_t(x)\Theta(x;q).
\end{equation*}
\end{definition}
\begin{remark}\label{rem:elltrVY} Note that $V_{t}(x)$ and $Y_{t}(x)$ have the following elliptic transformation properties: 
\begin{equation*}
V_{t}(qx) = x^{-2t+1}V_{t}(x) \ \ \text{and} \ \  Y_{t}(qx) = -x^{-2t}Y_{t}(x).
\end{equation*}
They can be directly obtained from the elliptic transformation property for the theta function
\begin{equation}\label{eq:elltrprop}
    \Theta(q^nx;q)=(-1)^{n}q^{-\binom{n}{2}}x^{-n}\Theta(x;q).
\end{equation}
\end{remark}

We can derive the following convenient form for the $Y_{t}(x)$.
\begin{proposition}\label{prop:formforY} We have
\begin{equation*}
Y_{t}(x) = \sum_{s=0}^{2t-1} (-1)^s q^{\binom{s}{2}}\Theta(q^{(2t-1)(t-s)}; q^{2t(2t-1)}) x^s \sum_{n = -\infty}^{\infty} q^{2t\binom{n}{2}} (-q^s x^{2t})^n.
\end{equation*}
\end{proposition}
\begin{proof}[Proof of Proposition \ref{prop:formforY}]
As $Y_{t}(x)$ has no poles in $x\in \C\smallsetminus \{0\}$ we can consider its Laurent series $Y_{t}(x) = \sum_{r} a_r x^r$. Using Remark \ref{rem:elltrVY} we can obtain the relation on the coefficients:
\begin{equation*}
q^r a_r = -a_{r+2t}.
\end{equation*}
We are able to iterate the previous formula and obtain
\begin{equation*}
a_{2tn+s} = (-1)^nq^{2t\binom{n}{2}+ns}a_s.
\end{equation*}
Then using this relation we have
\begin{align*}
Y_{t}(x) = \sum_{r = -\infty}^{\infty}a_r x^r 
&= \sum_{s=0}^{2t-1}\sum_{n = -\infty}^{\infty} a_{2tn+s}x^{2tn+s} \\
&= \sum_{s=0}^{2t-1} a_s x^s \sum_{n = -\infty}^{\infty} (-1)^nq^{2t\binom{n}{2}+ns} x^{2tn} \\
&= \sum_{s=0}^{2t-1} a_s x^s \sum_{n = -\infty}^{\infty} q^{2t\binom{n}{2}} (-q^s x^{2t})^n.
\end{align*}
We can also find the coefficients $a_s$ explicitly. From
\begin{equation*}
Y_{t}(x) = \Theta(-x^{2t-1};q^{2t-1})\Theta(x;q)
\end{equation*}
we find
\begin{equation*}
\sum_{s = -\infty}^{\infty} a_s x^s = \sum_{k = -\infty}^{\infty}  q^{(2t-1)\binom{k}{2}}x^{(2t-1)k} \times \sum_{n = -\infty}^{\infty}  (-1)^nq^{\binom{n}{2}}x^{n}.
\end{equation*}
Hence we can conclude
\begin{align*}
a_s &= \sum_{r = -\infty}^{\infty} (-1)^{s-(2t-1)r} q^{(2t-1)\binom{r}{2}+\binom{s-(2t-1)r}{2}} \\
&= (-1)^s q^{\binom{s}{2}}\sum_{r = -\infty}^{\infty}(-1)^r q^{(2t-1)\binom{r}{2} - (2t-1)sr +\binom{(2t-1)r+1}{2}} \\
&= (-1)^s q^{\binom{s}{2}}\sum_{r = -\infty}^{\infty} (-1)^r q^{2t(2t-1)\binom{r}{2} +(2t-1)(t-s)r} = (-1)^s q^{\binom{s}{2}}\Theta(q^{(2t-1)(t-s)}; q^{2t(2t-1)}).
\end{align*}
Comparing the two calculations above yields the desired result:
\begin{equation*}
Y_{t}(x) = \sum_{s=0}^{2t-1} (-1)^s q^{\binom{s}{2}}\Theta(q^{(2t-1)(t-s)}; q^{2t(2t-1)}) x^s \sum_{n = -\infty}^{\infty} q^{2t\binom{n}{2}} (-q^s x^{2t})^n.\qedhere
\end{equation*}
\end{proof}
\begin{proof}[Proof of Theorem \ref{theorem:newcalc}]
Let us consider the function
\begin{equation*}
G(x) = V_{t}(x)f_{t,m}(x)
\end{equation*}
and find the functional equation for it using the functional equation for $f_{t,m}(x)$ from Section \ref{sect:functeqF}:
\begin{align*}
G(qx) &= V_{t}(qx)f_{t,m}(qx) \\
&= x^{-2t+1}V_{t}(x)\\ 
&\qquad \times \left(-x^{2t-1} f_{t,m}(x) + x^{2t-m-1} + x^{m} - \frac{\Theta(x;q)}{(q)_{\infty}^3} \sum_{k=0}^{2t-1} x^{k-1}q^{k}\theta_{t+k,m+k}\right)\\
&=
-G(x) + x^{-m}V_{t}(x) + x^{-2t+m+1}V_{t}(x) - \frac{x^{-2t}Y_{t}(x)}{(q)_{\infty}^3} \sum_{k=0}^{2t-1} x^k q^{k}\theta_{t+k,m+k}. 
\end{align*}
By definition of $V_{t}(x)$ and Proposition \ref{prop:formforY} we have
\begin{multline}\label{eq:forGqxGx}
G(qx) = -G(x) + \sum_{r = -\infty}^{\infty}  q^{(2t-1)\binom{r}{2}}x^{(2t-1)r-m} + \sum_{r = -\infty}^{\infty}  q^{(2t-1)\binom{r}{2}}x^{(2t-1)r-2t+m+1} \\
- \frac{x^{-2t}}{(q)_{\infty}^3} \sum_{k=0}^{2t-1} x^kq^{k}\theta_{t+k,m+k}\sum_{s=0}^{2t-1} a_s x^s \sum_{r = -\infty}^{\infty}(-1)^r q^{2t\binom{r}{2}} (q^s x^{2t})^r.
\end{multline}
Now let us consider the last term in the previous sum. We find
\begin{align*}
\frac{x^{-2t}}{(q)_{\infty}^3} &\sum_{k=0}^{2t-1} x^kq^{k}\theta_{t+k,m+k}\sum_{s=0}^{2t-1} a_s x^s \sum_{r = -\infty}^{\infty} q^{2t\binom{r}{2}} (-q^s x^{2t})^r
\\
&= \frac{1}{(q)_{\infty}^3} \sum_{r = -\infty}^{\infty}\sum_{s=0}^{2t-1}\sum_{k=0}^{2t-1}(-1)^r q^{k}\theta_{t+k,m+k} a_s  q^{2t\binom{r}{2}} q^{rs} x^{2tr+k+s-2t}.
\end{align*}
Taking $2tr+s=d$, $R(d) = \lfloor \frac{d}{2t} \rfloor$, $S(d) = d- 2t \lfloor \frac{d}{2t} \rfloor$ we obtain
\begin{gather*}
\frac{1}{(q)_{\infty}^3} \sum_{d = -\infty}^{\infty}\sum_{k=0}^{2t-1} (-1)^{R(d)}q^{k}\theta_{t+k,m+k} a_{S(d)}  q^{2t\binom{R(d)}{2}} q^{R(d)S(d)} x^{d+k-2t}.
\end{gather*}
Change the variable $d \to d-k+2t$ we have
\begin{gather*}
\frac{1}{(q)_{\infty}^3} \sum_{d = -\infty}^{\infty}\sum_{k=0}^{2t-1} (-1)^{R(d-k+2t)}q^{k}\theta_{t+k,m+k} a_{S(d-k+2t)}  q^{2t\binom{R(d-k+2t)}{2}} q^{R(d-k+2t)S(d-k+2t)} x^{d}.
\end{gather*}
As $G(x)$ has no poles, we are able to take its Laurent series $G(x) = \sum a_r x^r$. Then by formula (\ref{eq:forGqxGx}) and the previous expression for the last term we find 
\begin{equation*}
q^d a_d = -a_d + C^{(1)}_d + C^{(2)}_d - B_d \\
\end{equation*}
where
\begin{align*}
&C^{(1)}_d = \begin{cases}
    q^{(2t-1)\binom{k}{2}} & \textup{if} \ d = (2t-1)k-m \ \text{with} \ k \in \Z,\\
   0 & \text{otherwise},
\end{cases}\\
&C^{(2)}_d = \begin{cases}
    q^{(2t-1)\binom{k}{2}} & \textup{if} \ d = (2t-1)k-2t+m+1 \ \text{with} \ k \in \Z,\\
   0 & \text{otherwise},
\end{cases}\\
&B_d = \frac{1}{(q)_{\infty}^3} \sum_{k=0}^{2t-1}(-1)^{R(d-k+2t)} q^{k}\theta_{t+k,m+k} a_{S(d-k+2t)}  q^{2t\binom{R(d-k+2t)}{2}} q^{R(d-k+2t)S(d-k+2t)}.
\end{align*}
Hence by moving $-a_r$ to the left, dividing the equation by $(1+q^r)$ and reformulating the last term we obtain
\begin{multline*}
G(x) = \sum_{r = -\infty}^{\infty} \frac{ q^{(2t-1)\binom{r}{2}}x^{(2t-1)r-m}}{(1+q^{(2t-1)r-m})} + \sum_{r = -\infty}^{\infty} \frac{ q^{(2t-1)\binom{r}{2}}x^{(2t-1)r-2t+m+1}}{(1+q^{(2t-1)r-2t+m+1})} -\\- \frac{x^{-2t}}{(q)_{\infty}^3} \sum_{k=0}^{2t-1} x^k q^{k}\theta_{t+k,m+k}\sum_{s=0}^{2t-1} a_s x^s\sum_{r = -\infty}^{\infty} \frac{(-1)^r q^{2t\binom{r}{2}} (q^s x^{2t})^r}{(1+q^{2tr+k+s-2t})}.
\end{multline*}
Recalling the definition of the Appell function (\ref{eq:appell}), we can write
\begin{align*}
G(x) &= x^{-m}\Theta(-x^{2t-1};q^{2t-1})m(x^{-2t+1}q^{2t-1-m},-x^{2t-1};q^{2t-1}) \\
&\qquad + x^{-2t+m+1}\Theta(-x^{2t-1};q^{2t-1})m(x^{-2t+1}q^{m},-x^{2t-1};q^{2t-1}) \\
 &\qquad - \frac{x^{-2t}}{(q)_{\infty}^3}\sum_{k=0}^{2t-1} x^k q^{k}\theta_{t+k,m+k}\\
 &\qquad \qquad \times \sum_{s=0}^{2t-1} (-1)^s q^{\binom{s}{2}}\Theta(q^{(2t-1)(t-s)}; q^{2t(2t-1)}) x^s  \Theta(q^sx^{2t};q^{2t})m(-q^{k}x^{-2t}, q^s x^{2t}; q^{2t}).
\end{align*}
In conclusion, we divide the equation by $V_{t}(x)$ and find
\begin{align*}
f_{t,m}(x) &= x^{-2t+m+1} m(x^{-2t+1}  q^{m},-x^{2t-1}; q^{2t-1})
+x^{-m}   m(x^{-2t+1}q^{2t-1-m} ,-x^{2t-1};q^{2t-1})  \\
&\qquad  -\frac{x^{-2t}}{\Theta(-x^{2t-1};q^{2t-1})(q)_{\infty}^3}  \sum_{k=0}^{2t-1} x^k q^{k}\theta_{t+k,m+k} \\
&\qquad \qquad  \times\sum_{s=0}^{2t-1} (-1)^s q^{\binom{s}{2}}\Theta(q^{(2t-1)(t-s)}; q^{2t(2t-1)}) x^s  \Theta(q^sx^{2t};q^{2t})m(-q^{k}x^{-2t}, q^s x^{2t}; q^{2t}).\qedhere
\end{align*}
\end{proof}

\section{Application of Theorem \ref{theorem:main}: interpretation of $\mathcal{U}_r^{(m)}(x;q)$ in terms of Hecke-type double-sums}\label{sect:applicationmainmt}
First, let us interpret the indefinite binary theta series
\begin{equation*}
\theta_{p,m} = \sum_{r\equiv s \Mod 2} \sg(r,s) (-1)^{\frac{r-s}{2}}q^{\frac{r^2}{8}+\frac{4t-1}{4}rs+\frac{s^2}{8}+\frac{4p-1-2m}{4}r+\frac{1+2m}{4}s}
\end{equation*}
in terms of Hecke-type double-sums (\ref{eq:fabc}).
\begin{proposition} \label{prop:vpmintermsfabc}
We have
\begin{equation*}
\theta_{p,m} = f_{1,4t-1,1}(q^{2p-m}, q^{1+m}; q) + q^{t+p}f_{1,4t-1,1}(q^{2p+2t-m}, q^{2t+m+1}; q).
\end{equation*}
\end{proposition}
\begin{proof}[Proof of Proposition \ref{prop:vpmintermsfabc}]
Change the variables $(r,s)$ to $(2r,2s)$, $(2r+1,2s+1)$:
\begin{equation*}
\theta_{p,m} = \left(\sum_{r\equiv s \equiv 0 \Mod 2} + \sum_{r\equiv s \equiv 1 \Mod 2}\right)\sg(r,s) (-1)^{\frac{r-s}{2}}q^{\frac{r^2}{8}+\frac{4t-1}{4}rs+\frac{s^2}{8}+\frac{4p-1-2m}{4}r+\frac{1+2m}{4}s}.
\end{equation*}
Let us consider the first sum
\begin{align*}
\sum_{r\equiv s \equiv 0 \Mod 2}&\sg(r,s) (-1)^{\frac{r-s}{2}}q^{\frac{r^2}{8}+\frac{4t-1}{4}rs+\frac{s^2}{8}+\frac{4p-1-2m}{4}r+\frac{1+2m}{4}s}\\
&=\sum_{r,s \in \Z}\sg(r,s) (-1)^{r-s}q^{\frac{r^2}{2}+(4t-1)rs+\frac{s^2}{2}+\frac{4p-1-2m}{2}r+\frac{1+2m}{2}s} \\
&=\sum_{r,s \in \Z}\sg(r,s) (-1)^{r-s}q^{\binom{r}{2}+(4t-1)rs+\binom{s}{2}+(2p-m)r+(m+1)s}\\
&=
f_{1,4t-1,1}(q^{2p-m}, q^{1+m}; q).
\end{align*}
Considering the second sum yields
\begin{align*}
\sum_{r\equiv s \equiv 1 \Mod 2}&\sg(r,s) (-1)^{\frac{r-s}{2}}q^{\frac{r^2}{8}+\frac{4t-1}{4}rs+\frac{s^2}{8}+\frac{4p-1-2m}{4}r+\frac{1+2m}{4}s}\\
&= \sum_{r,s \in \Z}\sg(r,s) (-1)^{r-s}q^{\frac{(2r+1)^2}{8}+\frac{4t-1}{4}(2r+1)(2s+1)+\frac{(2s+1)^2}{8}+\frac{4p-1-2m}{4}(2r+1)+\frac{1+2m}{4}(2s+1)}\\
&= \sum_{r,s \in \Z}\sg(r,s) (-1)^{r-s}q^{\frac{r^2}{2}+ \frac{r}{2}+ \frac{1}{8}+(4t-1)rs+\frac{4t-1}{2}r+\frac{4t-1}{2}s+\frac{4t-1}{4}+\frac{s^2}{2}+ \frac{s}{2}+ \frac{1}{8}}\\
&\qquad \qquad \cdot q^{(\frac{4p-1-2m}{2})r+\frac{4p-1-2m}{4}+ (\frac{1+2m}{2})s+\frac{1+2m}{4}}\\
&=  q^{t+p}\sum_{r,s \in \Z}\sg(r,s) (-1)^{r-s}q^{\binom{r}{2} +(4t-1)rs + \binom{s}{2} +(2t+2p-m)r +(2t+m+1)s}\\
&=q^{t+p}f_{1,4t-1,1}(q^{2p+2t-m}, q^{2t+m+1}; q).
\end{align*}
To obtain the final result we sum the two pieces.
\end{proof}

\begin{proof}[Proof of Corollary \ref{cor:main}]
The proof is a straightforward application of Proposition \ref{prop:vpmintermsfabc} and identity (\ref{eq:anotherformmain}).
\end{proof}

Remind the reader of Remark \ref{rem:UandF}:
\begin{equation*}
\mathcal{U}_{t}^{(m)}(-x;q)=f_{t+1,m}(x).
\end{equation*}

We see that Corollary \ref{cor:main} gives us a representation of $\mathcal{U}_t^{(m)}(x;q)$ in terms of Hecke-type double-sums.

We can consider cases of $\mathcal{U}^{(m)}_2(x;q)$, $\mathcal{U}^{(m)}_3(x;q)$. 
From Definition \ref{def:Umt} we have
\begin{gather*}
\mathcal{U}_{2}^{(1)}(x;q) := \sum_{n\geq 0} (-x)_{n}(-\frac{q}{x})_{n} q^{n} \sum_{k = 0}^{n} q^{k^2+k}\myatop{n+k}{n-k}_q,\\
\mathcal{U}_{2}^{(2)}(x;q) := \sum_{n\geq 0} (-x)_{n}(-\frac{q}{x})_{n} q^{n} \sum_{k = 0}^{n} q^{k^2+k}\myatop{n+k+1}{n-k}_q.
\end{gather*}
Corollary \ref{cor:main} gives us
\begin{multline*}
\mathcal{U}_{2}^{(1)}(-x;q) = \frac{1}{(q)^3_{\infty}} \sum_{k=0}^{5}(-1)^k q^{\binom{k}{2}+k} (f_{1,11,1}(q^{5+k}, q^{2+k}; q) + q^{6+k}f_{1,11,1}(q^{11+k}, q^{8+k}; q)) \\ \times f_{1,6,30}(x^{-1}q^{k+1}, -q^{5(k+3)}; q),
\end{multline*}
\begin{multline*}
\mathcal{U}_{2}^{(2)}(-x;q) = \frac{1}{(q)^3_{\infty}} \sum_{k=0}^{5}(-1)^k q^{\binom{k}{2}+k} (f_{1,11,1}(q^{4+k}, q^{3+k}; q) + q^{6+k}f_{1,11,1}(q^{10+k}, q^{9+k}; q)) \\ \times f_{1,6,30}(x^{-1}q^{k+1}, -q^{5(k+3)}; q).
\end{multline*}

Similarly for $\mathcal{U}^{(m)}_3(x;q)$ we have
\begin{gather*}
\mathcal{U}_{3}^{(1)}(x;q) := \sum_{n\geq 0} (-x)_{n}(-\frac{q}{x})_{n} q^{n} \sum_{k = 0}^{n} \sum_{j = 0}^{k} q^{k^2+k} q^{j^2+j}\myatop{k+j}{k-j}_q\myatop{n+k+2j}{n-k}_q,\\
\mathcal{U}_{3}^{(2)}(x;q) := \sum_{n\geq 0} (-x)_{n}(-\frac{q}{x})_{n} q^{n} \sum_{k = 0}^{n} \sum_{j = 0}^{k} q^{k^2+k} q^{j^2+j}\myatop{k+j+1}{k-j}_q\myatop{n+k+2j+1}{n-k}_q,\\
\mathcal{U}_{3}^{(3)}(x;q) := \sum_{n\geq 0} (-x)_{n}(-\frac{q}{x})_{n} q^{n} \sum_{k = 0}^{n} \sum_{j = 0}^{k} q^{k^2+k} q^{j^2+j}\myatop{k+j+1}{k-j}_q\myatop{n+k+2j+2}{n-k}_q.
\end{gather*}
Corollary \ref{cor:main} gives us
\begin{multline*}
\mathcal{U}_{3}^{(1)}(-x;q) = \frac{1}{(q)^3_{\infty}}\sum_{k=0}^{7}(-1)^k q^{\binom{k}{2}+k} (f_{1,15,1}(q^{7+k}, q^{2+k}; q) + q^{8+k}f_{1,15,1}(q^{15+k}, q^{10+k}; q)) \\ \times f_{1,8,56}(x^{-1}q^{k+1}, -q^{7(k+4)}; q),
\end{multline*}
\begin{multline*}
\mathcal{U}_{3}^{(2)}(-x;q) = \frac{1}{(q)^3_{\infty}}\sum_{k=0}^{7}(-1)^k q^{\binom{k}{2}+k} (f_{1,15,1}(q^{6+k}, q^{3+k}; q) + q^{8+k}f_{1,15,1}(q^{14+k}, q^{11+k}; q)) \\ \times f_{1,8,56}(x^{-1}q^{k+1}, -q^{7(k+4)}; q),
\end{multline*}
\begin{multline*}
\mathcal{U}_{3}^{(3)}(-x;q) = \frac{1}{(q)^3_{\infty}}\sum_{k=0}^{7}(-1)^k q^{\binom{k}{2}+k} (f_{1,15,1}(q^{5+k}, q^{4+k}; q) + q^{8+k}f_{1,15,1}(q^{13+k}, q^{12+k}; q)) \\ \times f_{1,8,56}(x^{-1}q^{k+1}, -q^{7(k+4)}; q). 
\end{multline*}

\section{Application of Theorem \ref{theorem:main}: interpretation of $\mathcal{U}(x;q)$ in terms of Hecke-type double-sums}\label{sect:applicationmain11}
We have defined $\mathcal{U}(x;q)$ as a special case of vector-valued $\mathcal{U}^{(m)}_{t}(x;q)$ with $t=m=1$. Remark \ref{rem:UandF} tells us that
\begin{equation*}
\mathcal{U}(-x;q)= f_{2,1}(x).
\end{equation*}
So let us study the properties of $f_{2,1}(x)$. Firstly, note that by Proposition \ref{prop:thetaproperties} we can derive that $\theta_{2,1} = \theta_{4,3} = 0$, $\theta_{3,2} = -q^{-1}\theta_{1,1}$, $\theta_{5,4} = q^{-3}\theta_{1,1}$. Thus by Theorem \ref{theorem:main}:
\begin{align*}
f_{2,1}(x) &= \frac{q^{-\frac{3}{8}}\theta_{1,1}}{(q)^3_{\infty}} \left( \sum_{\substack{r,l \in \Z\\ l\equiv 1 \Mod {4}}}  - \sum_{\substack{r,l \in \Z\\ l\equiv 3 \Mod {4}}}\right) \sg(r,l)(-1)^r q^{\frac{r^2}{2}+rl+\frac{3}{8}l^2+\frac{1}{2}r}x^{-r} \\&= \frac{\theta_{1,1}}{(q)^3_{\infty}} \sum_{r,l \in \Z} \sg(r,l)(-1)^{r+l} q^{\frac{r^2}{2}+2rl+\frac{3}{2}l^2+\frac{3}{2}r+\frac{3}{2}l}x^{-r}.
\end{align*}
Hence $f_{2,1}(x)$ in terms of Hecke-type double-sums reads
\begin{equation*}
f_{2,1}(x) =
\frac{\theta_{1,1}}{(q)^3_{\infty}} f_{1,2,3}(x^{-1}q^2,q^3;q).
\end{equation*}
Now we want to calculate $\theta_{1,1}$. To do it, will use the following interesting identity coming from the proof of Theorem \ref{theorem:main}:
\begin{lemma}\label{lemma:interstingidentity} For $t\geq 2$, $1\leq m <t$ and $l \in\Z$ we have
\begin{align*}
  \sum_{k=0}^{2t-1} (-1)^{k} \theta_{t+k,m+k}^* \times \sum_{\substack{r\in \Z\\ r \equiv 2tl+(2t-1)k \Mod{2t(2t-1)}}}&q^{\frac{1}{4t(2t-1)}r^2}\\
  &= 
 \begin{cases}
   (-1)^{m} q^{\frac{1}{8}} (q)^3_{\infty} & \textup{if} \ l \equiv m \Mod {2t-1}, \\
   (-1)^{m+1} q^{\frac{1}{8}} (q)^3_{\infty} & \textup{if} \ l \equiv -m \Mod {2t-1}, \\
   0 & \textup{otherwise}.
\end{cases}
\end{align*}
\end{lemma}
\begin{proof}[Proof of Lemma \ref{lemma:interstingidentity}]
By Definitions \ref{eq:seqb} and \ref{eq:seqc}, we have
\begin{equation*}
b_r =  q^{\frac{1}{2(2t-1)}r^2 + \frac{1}{2}r + \binom{r+1}{2}} \frac{1}{(q)^3_{\infty}} \sum_{k=0}^{2t-1} (-1)^{k} q^{\binom{k}{2}+k+r(k-1)} \theta_{t+k,m+k}
\end{equation*}
and
\begin{equation*}
C_r = \begin{cases}
   (-1)^{r} q^{\frac{1}{2(2t-1)}r^2 + \frac{1}{2}r} & \textup{if} \ r \in \{-m, -2t+m+1\},\\
   0 & \textup{if} \ r \notin \{-m, -2t+m+1\}.
\end{cases}
\end{equation*}
Then from the proof of Lemma \ref{lemma:solverectool}, we have
\begin{align*}
\sum_{n \equiv r \Mod {2t-1}}b_n &= -\sum_{n \equiv r \Mod {2t-1}} C_n\\
&= \begin{cases}
   (-1)^{m} q^{\frac{1}{2(2t-1)}(m-2t+1)^2 + \frac{1}{2}(m-2t+1)} & \textup{if} \ r \equiv m \Mod {2t-1}, \\
   (-1)^{m+1} q^{\frac{1}{2(2t-1)}m^2 - \frac{1}{2}m} & \textup{if} \ r \equiv -m \Mod {2t-1}, \\
   0 & \textup{otherwise},
\end{cases}\\
&=
 \begin{cases}
   (-1)^{m} q^{\frac{1}{2(2t-1)}m^2 - \frac{1}{2}m} & \textup{if} \ l \equiv m \Mod {2t-1}, \\
   (-1)^{m+1} q^{\frac{1}{2(2t-1)}m^2 - \frac{1}{2}m} & \textup{if} \ l \equiv -m \Mod {2t-1}, \\
   0 & \text{otherwise}.
\end{cases}
\end{align*}
We also have
\begin{align*}
&\sum_{n \equiv l \Mod {2t-1}}b_n \\
&\qquad = \sum_{r\in \Z} b_{(2t-1)r+l}\\
&\qquad = \frac{1}{(q)^3_{\infty}} \sum_{r\in \Z} q^{\frac{1}{2(2t-1)}((2t-1)r+l)^2 + \frac{1}{2}((2t-1)r+l) + \binom{(2t-1)r+l+1}{2}} \\ &\qquad \qquad \times \sum_{k=0}^{2t-1} (-1)^{k}q^{\binom{k}{2}+k+((2t-1)r+l)(k-1)} \theta_{t+k,m+k}\\
&\qquad = 
\frac{1}{(q)^3_{\infty}} \sum_{k=0}^{2t-1} (-1)^{k} \theta_{t+k,m+k}  \sum_{r\in \Z}q^{t(2t-1)r^2+2trl+\frac{t}{2t-1}l^2+(2t-1)r+l+\binom{k}{2}+k+(2t-1)rk-(2t-1)r+lk-l}\\
&\qquad =\frac{1}{(q)^3_{\infty}} \sum_{k=0}^{2t-1} (-1)^{k} \theta_{t+k,m+k}  \sum_{r\in \Z}q^{t(2t-1)r^2+2trl+\frac{t}{2t-1}l^2+\binom{k}{2}+k+(2t-1)rk+lk}.
\end{align*}
Using the identities
\begin{equation*}
C_{t+k,m+k} = \frac{1}{4t}k^2+\frac{k}{2}+\frac{m}{2}-\frac{m^2}{2(2t-1)}+\frac{1}{8}
\end{equation*}
and
\begin{align*}
&\frac{1}{4t(2t-1)}(2t(2t-1)r+2tl+(2t-1)k)^2 \\
&\qquad \qquad = t(2t-1)r^2+\frac{t}{2t-1}l^2+\frac{2t-1}{4t}k^2+2trl+(2t-1)rk+lk,
\end{align*}
we can rewrite the exponent of $q$ in the last term as
\begin{align*}
 &t(2t-1)r^2+2trl+\frac{t}{2t-1}l^2+\binom{k}{2}+k+(2t-1)rk+lk\\
 &\qquad \qquad = C_{t+k,m+k} + \frac{1}{4t(2t-1)}(2t(2t-1)r+2tl+(2t-1)k)^2 -\frac{m}{2}+\frac{m^2}{2(2t-1)}-\frac{1}{8}.
\end{align*}
Using the fact that  $\theta^{*}_{p,m} = q^{C_{p,m}}\theta_{p,m}$, we are able to produce
\begin{equation*}
\sum_{n \equiv l \Mod {2t-1}}b_n = \frac{q^{-\frac{m}{2}+\frac{m^2}{2(2t-1)}-\frac{1}{8}}}{(q)^3_{\infty}} \sum_{k=0}^{2t-1} (-1)^{k} \theta_{t+k,m+k}^*  \sum_{r\in \Z}q^{\frac{1}{4t(2t-1)}(2t(2t-1)r+2tl+(2t-1)k)^2}.
\end{equation*}
The above calculation yields
\begin{align*}
 \frac{q^{-\frac{m}{2}+\frac{m^2}{2(2t-1)}-\frac{1}{8}}}{(q)^3_{\infty}} &\sum_{k=0}^{2t-1} (-1)^{k} \theta_{t+k,m+k}^*  \sum_{\substack{r\in \Z\\ r \equiv 2tl+(2t-1)k \Mod{2t(2t-1)}}}q^{\frac{1}{4t(2t-1)}r^2}\\
 &= 
 \begin{cases}
   (-1)^{m} q^{\frac{1}{2(2t-1)}m^2 - \frac{1}{2}m} & \textup{if} \ l \equiv m \Mod {2t-1}, \\
   (-1)^{m+1} q^{\frac{1}{2(2t-1)}m^2 - \frac{1}{2}m} & \textup{if} \ l \equiv -m \Mod {2t-1}, \\
   0 & \text{otherwise}.
\end{cases}
\end{align*}
Simplifying yields
\begin{equation*}
  \sum_{k=0}^{2t-1} (-1)^{k} \theta_{t+k,m+k}^*  \sum_{\substack{r\in \Z\\ r \equiv 2tl+(2t-1)k \Mod{2t(2t-1)}}}q^{\frac{1}{4t(2t-1)}r^2} = 
 \begin{cases}
   (-1)^{m} q^{\frac{1}{8}} (q)^3_{\infty} & \textup{if} \ l \equiv m \Mod {2t-1}, \\
   (-1)^{m+1} q^{\frac{1}{8}} (q)^3_{\infty} & \textup{if} \ l \equiv -m \Mod {2t-1}, \\
   0 & \text{otherwise}.
\end{cases}\qedhere
\end{equation*}
\end{proof}
Note that $\theta_{2,1}^*=0$, $\theta_{3,2}^*=-\theta_{1,1}^*$, $\theta_{4,3}^*=0$, $\theta_{5,4}^*=\theta_{1,1}^*$. Hence we can apply Lemma \ref{lemma:interstingidentity} to calculate $\theta_{1,1}$ taking $t=2$, $m=1$, $l=1$:
\begin{equation*}
 \theta_{1,1}^* \times \left(-\sum_{\substack{r \in \Z\\r\equiv 7 \Mod{12}}}+\sum_{\substack{r \in \Z\\r\equiv 1 \Mod{12}}}\right)q^{\frac{1}{24}r^2} = q^{\frac{1}{8}}(q)_{\infty}^3.
\end{equation*}
Recall the definition of the $\eta$-function:
\begin{equation*}
 \eta(q) := \sum_{r \in \Z} \chi(r)q^{\frac{1}{24}r^2} = q^{\frac{1}{24}}(q)_{\infty},
\end{equation*}
where $\chi(r)$ is an even Dirichlet character modulus $12$, such that $\chi(\pm 1)=1$, $\chi(\pm 5)=-1$.
Thus we have
\begin{gather*}
 \theta_{1,1}^* = \eta^2,\\
 \theta_{1,1} = (q)^2_{\infty}.
\end{gather*}
Finally we obtain
\begin{equation} \label{eq:Uandf21andf123}
\mathcal{U}(-x;q) = f_{2,1}(x) = \frac{1}{(q)_{\infty}} f_{1,2,3}(x^{-1}q^2,q^3;q).
\end{equation}
\begin{remark} We are also able to obtain identity (\ref{eq:Uandf21andf123}) as a corollary of equation \cite[(4.17)]{M1}:
\begin{equation}\label{eq:417ident}
(q^2;q^2)_{\infty}\sum_{n=0}^{\infty}q^{2n+1}(-aq;q^2)_n(-q/a;q^2)_n = qf_{3,2,1}(q^6,-aq^3;q^2)
\end{equation}
To see this we take $q \to q^{\frac{1}{2}}$ and $a=-x^{-1}q^{\frac{1}{2}}$ in this equation and apply the Definition \ref{def:Umt} for $t=m=1$. Identity \ref{eq:Uandf21andf123} also derived in \cite[Section 1, Remark 3]{L} using Bailey pairs method. 
\end{remark}

\section{Application of Theorem \ref{theorem:newcalc}: Evaluations of the double-sum $f_{1,2,3}(x^{-1}q^2,q^3;q)$}\label{sect:applicationnewf123}
In this section we convert the Hecke-type double-sum $f_{1,2,3}(x^{-1}q^2,q^3;q)$ to different types of Appell function forms. The first expression we present does not contain theta terms.  For the second expression, we minimize the number of Appell functions. We begin by introducing more convenient notation. Define
\begin{gather*}
\Theta_{a,m}:=\Theta(q^a;q^m), \Bar{\Theta}_{a,m}:=\Theta(-q^a;q^m) \ \ \text{and} \ \ \Theta_{m}:=\Theta_{m,3m}=(q^{m};q^{m})_{\infty}.
\end{gather*}
We give the first expression in the following proposition.
\begin{proposition} \label{prop:newcalcf21} We have
\begin{align*}
f_{1,2,3}(x^{-1}q^2,q^3;q) &= x^{-1}m(x^{-3}q^2,-x^{3};q^{3})(q)_{\infty} + x^{-2}m(x^{-3}q,-x^{3};q^{3})(q)_{\infty} \\
&\qquad + x^{-1}\frac{\Theta_{6,12}\Theta(x^{4};q^{4})}{\Theta(-x^{3};q^{3})} [ x^{-2} m(-x^{-4}q, x^{4}; q^{4})- m(-x^{-4}q^3, x^{4}; q^{4})] \\
&\qquad - \frac{\Theta_{3,12}\Theta(qx^{4};q^{4})}{\Theta(-x^{3};q^{3})} [ x^{-2} m(-x^{-4}q, q x^{4}; q^{4})- m(-x^{-4}q^3, qx^{4}; q^{4})]\\
&\qquad - q^3x^{2}\frac{\Theta_{-3,12}\Theta(q^3x^{4};q^{4})}{\Theta(-x^{3};q^{3})}[  x^{-2}m(-x^{-4}q, q^3 x^{4}; q^{4})- m(-x^{-4}q^3, q^3 x^{4}; q^{4})].
\end{align*}
\end{proposition}
\begin{proof}[Proof of Proposition \ref{prop:newcalcf21}]
We need to just apply Theorem \ref{theorem:newcalc}
for $t=2$, $m=1$ and use identities $\theta_{2,1} = \theta_{4,3} = 0$, $\theta_{3,2} = -q^{-1}\theta_{1,1}$, $\theta_{5,4} = q^{-3}\theta_{1,1}$, $\theta_{1,1} = (q)_{\infty}^2$, $\Theta_{0,12}=0$. We obtain 
\begin{align*}
f_{1,2,3}(x^{-1}q^2,q^3;q) &= x^{-1}m(x^{-3}q^2,-x^{3};q^{3})(q)_{\infty} + x^{-2}m(x^{-3}q,-x^{3};q^{3})(q)_{\infty} \\
&\qquad + \frac{1}{\Theta(-x^{3};q^{3})}\sum_{s=0}^{3} (-1)^s q^{\binom{s}{2}}x^s\Theta_{6-3s,12} \Theta(q^s x^{4};q^{4})\\
& \qquad \qquad \times( x^{-3} m(-x^{-4}q, q^s x^{4}; q^{4})-x^{-1} m(-x^{-4}q^3, q^s x^{4}; q^{4})).\qedhere
\end{align*}
\end{proof}
Next we use Proposition \ref{prop:newcalcf21} to understand how $f_{1,2,3}(x^{-1}q^2,q^3;q)$ looks ``mod theta'', that is, we want to minimize the number of Appell functions in the Appell function form. 

\begin{corollary} \label{cor:thetamodid} We have
\begin{align*} 
f_{1,2,3}(x^{-1}q^2,q^3;q) &=  x^{-1}m(x^{-3}q^2,-x^{3};q^{3})(q)_{\infty} + x^{-2}m(x^{-3}q,-x^{3};q^{3})(q)_{\infty}\\
&\qquad- x^{-1}\Theta(x;q)m(qx^{-2},-1;q)
\\
&\qquad  - \frac{(q^2;q^2)_{\infty}^3\Theta_{1,2}\Theta(-x;q)\Theta(x^{-3};q^3)}{\Bar{\Theta}_{0,2}\Bar{\Theta}_{1,2}\Theta(-x^{3};q^3)\Theta(-x^{-2};q)}.
\end{align*}
\end{corollary} 
\begin{remark} \label{rem:thetamod} Recall the definition of the universal mock theta function $g(x;q)$:
\begin{equation} \label{eq:unimockg}
 g(x;q) := x^{-1}\left(-1+\sum_{n=0}^{\infty} \frac{q^{n^2}}{(x)_{n+1}(q/x)_n} \right)
\end{equation}
By [\cite{HM}, Proposition 4.2] it can be expressed in terms of Appell functions:
\begin{equation*} 
g(x;q) = -x^{-1}m(x^{-3}q^2,zx^{3};q^{3}) - x^{-2}m(x^{-3}q,zx^{3};q^{3})+\frac{(q)_{\infty}^2\Theta(xz;q)\Theta(z;q^3)}
{\Theta(x;q)\Theta(z;q)\Theta(x^3z;q^3)}.
\end{equation*}
We take $z=-1$ and find
\begin{equation*}
g(x;q) = -x^{-1}m(x^{-3}q^2,-x^{3};q^{3}) - x^{-2}m(x^{-3}q,-x^{3};q^{3})+\frac{(q)_{\infty}^2\Theta(-x;q)\Theta(-1;q^3)}
{\Theta(x;q)\Theta(-1;q)\Theta(-x^3;q^3)}.
\end{equation*}
Hence Corollary \ref{cor:thetamodid} tells us that
\begin{equation*}
f_{1,2,3}(x^{-1}q^2,q^3;q) \sim -(q)_{\infty} \cdot g(x;q)- x^{-1}\Theta(x;q)m(qx^{-2},-1;q).
\end{equation*}
where $\sim$ means equivalence of two objects ``mod theta''\cite{HM}.

\begin{proof}[Proof of Corollary \ref{cor:thetamodid}]
We use an identity from [\cite{M2013}, Corollary 3.10]:
\begin{multline*}
m(x,z;q) = m(-qx^2,z';q^4) -q^{-1}xm(-q^{-1}x^2,z';q^4)\\
+\frac{z'(q^2;q^2)_{\infty}^3}{\Theta(xz;q)\Theta(z';q^4)}\Big[\frac{\Theta(-qx^2zz';q^2)\Theta(z^2(z')^{-1};q^4)}{\Theta(-qx^2z';q^2)\Theta(z;q^2)}\\-\frac{\Theta(-q^2x^2zz';q^2)\Theta(q^2z^2(z')^{-1};q^4)}{\Theta(-qx^2z';q^2)\Theta(qz;q^2)}\Big].
\end{multline*}
Here we take $x \to qx^{-2}$, $z\to -1$ and for every term take the corresponding $z'\in \{x^4,qx^4,q^2x^4,q^3x^4\}$. Then we simplify the terms using identities $\Theta_{4,2}=\Theta_{6,2}=0$ and Proposition \ref{prop:formforY} for $t=2$:
\begin{equation*}
\frac{1}{\Theta(-x^{3};q^{3})}\sum_{s=0}^{3} (-1)^sq^{\binom{s}{2}}x^s \Theta(q^{3(2-s)};q^{12})\Theta(q^sx^4;q^4) = \Theta(x;q).
\end{equation*}
So we find
\begin{align*} 
f_{1,2,3}(x^{-1}q^2,q^3;q) &=  x^{-1}m(x^{-3}q^2,-x^{3};q^{3})(q)_{\infty} + x^{-2}m(x^{-3}q,-x^{3};q^{3})(q)_{\infty}\\
&\qquad- x^{-1}\Theta(x;q)m(qx^{-2},-1;q)
\\
&\qquad  +
\frac{(q^2;q^2)_{\infty}^3}{\Theta(-x^{3};q^{3})\Theta(-qx^{-2};q)}\Big[x^{3}\frac{\Theta_{6,12}\Theta_{3,2}}{\Bar{\Theta}_{0,2}\Bar{\Theta}_{3,2}}\Theta(x^{-4};q^4)
\\
&\qquad
-q^2x^2\frac{\Theta_{3,12}\Theta_{5,2}}{\Bar{\Theta}_{1,2}\Bar{\Theta}_{4,2}}\Theta(qx^{-4};q^4)-q^7x^4\frac{\Theta_{-3,12}\Theta_{7,2}}{\Bar{\Theta}_{1,2}\Bar{\Theta}_{6,2}}\Theta(q^{-1}x^{-4};q^4)\Big].
\end{align*} 
Then to get the desired result we need to use the elliptic transformation property (\ref{eq:elltrprop}) and the identity
\begin{equation*}
\Theta(x;q)\Theta(y;q^n) = \sum_{k=0}^{n}(-1)^kq^{\binom{k}{2}}x^k\Theta((-1)^nq^{\binom{n}{2}+kn}x^ny;q^{n(n+1)})\Theta(-q^{1-k}x^{-1}y;q^{n+1})
\end{equation*}
from \cite[Proposition 2.2, (2.4e)]{HM} with $n=3$, $x \to -x$ and $y=x^{-3}$.
\end{proof}
\end{remark}

\section{Another evaluation of the double-sum  $f_{1,2,3}(x^{-1}q^2,q^3;q)$} \label{sect:nothetaid}
Using \cite[Section 4.2]{M1} we can derive the Appell function form for $f_{1,2,3}(x^{-1}q^2,q^3;q)$ with the same Appell functions part as in Corollary \ref{cor:thetamodid}. Let us take identity \cite[(4.15)]{M1} and rewrite it in the following form:
\begin{equation*}
\sum_{n=0}^{\infty} q^{2n+1}(-aq;q^2)_n(-q/a;q^2)_n = -qg(-aq;q^2)+a\frac{\Theta(-aq;q^2)}{(q^2;q^2)_{\infty}}m(a^2,-1;q^2)-\frac{1}{2} \cdot\frac{a\Theta(aq;q^2)^3\Theta(a^2;q^4)}{\Theta^2_4\Theta(a^4;q^4)} 
\end{equation*}
where $g(x;q)$ is the universal mock theta function, defined in (\ref{eq:unimockg}). Let us substitute $q \to q^{\frac{1}{2}}$ and $a=-x^{-1}q^{\frac{1}{2}}$ into the previous equation, so we have
\begin{equation*}
\sum_{n=0}^{\infty} q^{n}(-x;q)_n(-q/x;q)_n = -g(q/x;q)-x^{-1}\frac{\Theta(q/x;q)}{(q)_{\infty}}m(qx^{-2},-1;q)+\frac{1}{2} \cdot \frac{x^{-1}\Theta(-q/x;q)^3\Theta(qx^{-2};q^2)}{(q^2;q^2)_{\infty}^2\Theta(q^2x^{-4};q^2)}. 
\end{equation*}
Using identity (\ref{eq:Uandf21andf123}), Definition \ref{def:Umt} and properties of $g(x;q)$ and $\Theta(x;q)$ we obtain
\begin{equation*}
f_{1,2,3}(x^{-1}q^2,q^3;q) = - (q)_{\infty} \cdot g(x;q)-x^{-1}\Theta(x;q)m(qx^{-2},-1;q)-\frac{1}{2}\cdot \frac{x^{-5}(q)_{\infty} \Theta(-x;q)^3\Theta(qx^{-2};q^2)}{(q^2;q^2)_{\infty}^2\Theta(x^{-4};q^2)}. 
\end{equation*}
This identity gives the same ``mod theta'' equivalence as stated in Remark \ref{rem:thetamod}.

If we take identity \cite[(4.16)]{M1}, we obtain a similar representation but it does not consist of theta terms:
\begin{equation*}
f_{1,2,3}(x^{-1}q^2,q^3;q) = -(q)_{\infty} \cdot g(x;q)-x^{-1}\Theta(x;q)m(qx^{-2},x;q).
\end{equation*}

\section*{Acknowledgements}
I would like to offer my gratitude to my scientific advisor Eric T. Mortenson for the suggestion to consider a new $U$-type function and also for helpful comments and observations. As well I want to thank Jeremy Lovejoy for his feedback.

\end{document}